\documentclass[a4paper]{gtart}

\usepackage[inner=30mm, outer=30mm, textheight=225mm]{geometry}

\usepackage{graphicx, subfigure, color}

\usepackage{amsmath, amssymb, latexsym, euscript}

\usepackage{mathptmx, wrapfig}

\usepackage{wrapfig,lscape,rotating}

\usepackage[small,bf]{caption}
\def\ne{n_{\emptyset}}
\def\nt{n_{\mathfrak{t}}}
\def\nq{n_{\mathfrak{q}}}

\def\Te{\Delta_{\emptyset}}
\def\Tt{\Delta_{\mathfrak{t}}}
\def\Tq{\Delta_{\mathfrak{q}}}

\def\lst{\mathbf{T}}
\def\tri{\mathcal{T}}

\def\even{\mathfrak{e}}
\def\odd{\mathfrak{o}}

\def\R{\mathbb{R}}
\def\Z{\mathbb{Z}}
\def\N{\mathbb{N}}

\DeclareMathOperator{\SL}{SL}

\newcommand{\abs}[1]{\lvert#1\rvert}

\theoremstyle{plain}
\newtheorem{theorem}{Theorem}
\newtheorem*{theorem*}{Theorem}
\newtheorem{lemma}[theorem]{Lemma}
\newtheorem{proposition}[theorem]{Proposition}
\newtheorem{corollary}[theorem]{Corollary}

\theoremstyle{definition}
\newtheorem{definition}[theorem]{Definition}
\newtheorem*{definition*}{Definition}

\theoremstyle{remark}
\newtheorem{remark}[theorem]{Remark}

\numberwithin{equation}{section}

\begin{document}

\title{$\Z_2$--Thurston Norm and Complexity of 3--Manifolds, II}
\author{William Jaco, J. Hyam Rubinstein, Jonathan Spreer and Stephan Tillmann}

\begin{abstract} 
In this sequel to earlier papers by three of the authors, we obtain a new bound on the complexity of a closed 3--manifold, as well as a characterisation of manifolds realising our complexity bounds. As an application, we obtain the first infinite families of minimal triangulations of Seifert fibred spaces modelled on Thurston's geometry $\widetilde{\SL_2(\R)}.$ 
\end{abstract}

\primaryclass{57Q15, 57N10; 57M50, 57M27}
\keywords{3--manifold, minimal triangulation, layered triangulation, efficient triangulation, complexity, Seifert fibred space, lens space}
\makeshorttitle


\section{Introduction}

Given a closed, irreducible 3--manifold, its complexity is the minimum number of tetrahedra in a (pseudo-simplicial) triangulation of the manifold. This number agrees with the complexity defined by Matveev~\cite{Matveev-complexity-1990} unless the manifold is $S^3,$ $\R P^3$ or $L(3,1).$ We denote the complexity of $M$ by $c(M).$ The only known infinite families of closed 3--manifolds for which an exact value of the complexity is known are spherical space forms \cite{Jaco-minimal-2009, Jaco-coverings-2011, Jaco-Z2-2013}. 

Asymptotic bounds for an infinite family of double branched coverings of certain alternating closed 3--braids are given by Ni and Wu~\cite{Ni-correction-2015} using the results of \cite{Jaco-Z2-2013}. New lower bounds on the complexity of closed 3--manifolds were recently given with different methods by Cha~\cite{Cha-complexity-2016, Cha-topological-2016}. 

In this paper, we continue to be interested in determining the exact complexity and minimal triangulations of infinite classes of closed 3--manifolds. In order to describe our results, we need the following definition.

Let $M$ be a closed, orientable, irreducible, connected 3--manifold, and let
$S$ be an embedded closed surface dual to a given $\varphi \in H^1(M;\Z_2).$ 
An analogue of Thurston's norm \cite{Thurston-norm-1986} is defined in \cite{Jaco-Z2-2013} as follows. If $S$ is connected, let $\chi_{-}(S) = \max \{ 0,-\chi(S)\},$ and otherwise let
$$\chi_{-}(S) = \sum_{S_i\subset S} \max \{ 0,-\chi(S_i)\},$$
where the sum is taken over all connected components of $S.$ Note that $S_i$ is not necessarily orientable. Define:
$$|| \ \varphi \ || = \min \{ \chi_{-}(S) \mid S \text{ dual to } \varphi\}.$$
The surface $S$ dual to $\varphi \in H^1(M;\Z_2)$ is said to be \emph{$\Z_2$--taut} if no component of $S$ is a sphere and $\chi(S) = -|| \ \varphi \ ||.$ As in \cite{Thurston-norm-1986}, one observes that (after possibly deleting compressible tori) every component of a $\Z_2$--taut surface is non-separating and geometrically incompressible. This follows from the fact that the Klein bottle does not embed in $\R P^3.$

Our previous papers \cite{Jaco-minimal-2009} and \cite{Jaco-Z2-2013} use this norm on $H^1(M;\Z_2)$ to obtain lower bounds on the complexity of 3--manifolds. The first main result (Theorem~\ref{thm:main 1}) in this paper generalises the work from \cite{Jaco-minimal-2009}, and the second main result (Theorem~\ref{thm:main 3}) improves the main result of \cite{Jaco-Z2-2013}. As applications of both results, we give the first infinite families of minimal triangulations of Seifert fibred spaces modelled on Thurston's geometry $\widetilde{\SL_2(\R)}.$


\subsection{Rank one}

In \cite{Jaco-minimal-2009} we exhibit an infinite family of lens spaces each member $L$ of which has even order fundamental group which satisfy the relationship $c(L) = 1 + 2 || \ \varphi \ ||,$ where $\varphi$ is the generator for $H^1(L;\Z_2).$ We call these lens spaces \emph{balanced} (since another characterisation, using the terminology of \S\ref{subsec:colouring}, is that the number of $\varphi$--even edges equals the number of $\varphi$--odd edges). An example is $L = L(2n, 1),$ where $|| \ \varphi \ || = n-2$ and $c(L) = 2n-3.$
 
Our first main result puts the main result of \cite{Jaco-minimal-2009} in a greater context. It can be summarised with the slogan: \emph{If $M$ contains a non-orientable $\Z_2$--taut surface of genus $n$, then $M$ is at least as complex as $L(2n, 1).$} An analogous statement can also be made for orientable $\Z_2$--taut surfaces. The proof of the below theorem does not use the main result from \cite{Jaco-minimal-2009}, but borrows heavily from its proof. 

\begin{theorem}\label{thm:main 1} 
Let $M$ be a closed orientable, irreducible, connected 3--manifold not homeomorphic with $\R P^3,$ and suppose that $0 \neq \varphi \in H^1(M;\Z_2).$ 
Then $c(M) \ge 1 + 2 || \ \varphi \ ||.$ Moreover, if equality holds, then $M$ is a balanced lens space.
\end{theorem}

\begin{corollary}
\label{cor:main}
Let $M$ be a closed orientable, irreducible, connected 3--manifold not homeomorphic with a balanced lens space and suppose that $0 \neq \varphi \in H^1(M;\Z_2).$ Then $c(M) \ge 2 + 2 || \ \varphi \ ||.$
\end{corollary}

We now describe a number of infinite families of triangulations that are minimal due to Corollary~\ref{cor:main}.

\textbf{Lens spaces.} This infinite family of triangulations contains lens spaces with fundamental group of even order, whose minimal layered triangulation has all $\varphi$--even edges of degree 4 except for either two $\varphi$--even edges or three $\varphi$--even edges. In the former family of triangulations, we have one edge of degree 3 and one of degree 5, and in the latter we have two edges of degree 3 and one of degree 6. To see that these triangulations are minimal, one checks that the minimal layered triangulation of such a lens space $L$ contains a unique (hence $\Z_2$--taut) non-orientable normal surface $S_\varphi$ representing the non-trivial class in $H^1(L;\Z_2).$ Denoting $\even$ the total number of even edges, then the number of tetrahedra in the minimal layered triangulation is $2\even,$ and the Euler characteristic of the taut surface is $1-\even.$ Hence $c(L) \le 2 \even = 2 + 2\even -2 = 2 + 2 || \ \varphi \ ||,$ which forces equality by Corollary~\ref{cor:main}. See \S\ref{subsec:lens spaces} for details and a description of these lens spaces.

\textbf{Small Seifert fibred spaces with a horizontal taut surface.} The Seifert fibred spaces
$$M_{k,m,n} = S^2( (1,1), (2k+1,1), (2m+1,1), (2n+1,1)),$$ 
where $k, m$ and $n$ are positive integers, are triangulated by \emph{augmented solid tori} having $2(k+m+n+1)$ tetrahedra. These triangulations were first described by Burton~\cite{regina}. We show that this is a minimal triangulation satisfying $ c(M_{k,m,n})  = 2 + 2|| \ \varphi \ ||$. The manifolds $M_{k,m,n} $ are modelled on Thurston's geometry $\widetilde{\SL_2(\R)}$, except for the $8$-tetrahedra case $k=m=n=1,$ which has a Nil structure. 
The minimality of this three-parameter family is explained in \S\ref{subsec:sSFS}.


\subsection{Rank two}

Under the additional hypothesis that $M$ is atoroidal, it is shown in \cite{Jaco-Z2-2013} that if $H \le H^1(M;\Z_2)$ is a subgroup of rank two, then 
$$c(M) \ge 2 + \sum_{0 \neq \varphi \in H} ||\;\varphi\;||.$$ 
The improved approach to the rank one case given in this paper, namely taking into account the existence of compression discs, similarly applies in the rank two case to show that the hypothesis that $M$ be atoroidal is superfluous. This has also been observed independently by Nakamura~\cite{Nakamura-complexity-2017}.
Our next main result adds an important observation to this, 
namely a characterisation of the minimal triangulations realising the above lower bound. This again allows us to give a simple corollary which determines further infinite families of minimal triangulations of small Seifert fibred spaces.
In this introduction, we give the following simplification of the technical Theorem~\ref{thm:main 3} stated in \S\ref{sec:rank 2}.

\begin{theorem}\label{thm:main 2} 
Let $M$ be a closed, orientable, irreducible, connected 3--manifold with the property that $H^1(M;\Z_2)$ has rank two. If $$c(M) = 2 + \sum_{0 \neq \varphi \in H^1(M;\Z_2)} ||\;\varphi\;||,$$ then 
$M$ is a generalised quaternionic space and the minimal triangulation is a twisted layered loop triangulation.
\end{theorem}

Every generalised quaternionic space with $H^1(M;\Z_2)$ of rank two satisfies the equation in the above theorem. These spaces and their triangulations are described in \cite{Jaco-coverings-2011}. As a consequence, we show that in the Pachner graph of all triangulations of a manifold, there may be triangulations with one tetrahedron more than the minimal triangulation but requiring arbitrarily many 2--3 and 3--2 moves to be reduced to a minimal triangulation. (See \S\ref{subsec:prism}.)

\begin{corollary}
\label{cor:3+}
Let $M$ be a closed, orientable, irreducible, connected 3--manifold such that $H^1(M;\Z_2)$ has rank two. If $M$ is not a generalised quaternionic space, then $$c(M) \ge 3 + \sum_{0 \neq \varphi \in H^1(M;\Z_2)} ||\;\varphi\;||.$$
\end{corollary}

\textbf{Small Seifert fibred spaces with three vertical taut surfaces.}
The Seifert fibred space
$$M'_{k,m,n} = S^2( (1,-1), (2k+2,1), (2m+2,1), (2n+2,1)),$$ 
where $k, m$ and $n$ are positive integers, is triangulated by an augmented solid torus having $2k+2m+2n+3$ tetrahedra. This has an $\widetilde{\SL_2(\R)}$ structure and satisfies $c(M'_{k,m,n})  = 3 + \sum || \ \varphi \ ||.$ See \S\ref{subsec:sSFS}.
In \S\ref{subsec:prism} we also describe an infinite family of minimal triangulations of prism manifolds realising the lower bound in Corollary~\ref{cor:3+}. 


\textbf{Acknowledgements.}
This research was supported through the programme ``Research in Pairs" by the Mathematisches Forschungsinstitut Oberwolfach in 2017. The authors would like to thank the staff at MFO for an excellent collaboration environment.
Jaco is partially supported by NSF grant DMS-1308767 and the Grayce B. Kerr Foundation.
Spreer is supported by the Einstein Foundation (project ``Einstein Visiting Fellow Santos'').
Research of Rubinstein and Tillmann is supported in part under the Australian Research Council's Discovery funding scheme (project number DP160104502). 


\section{Preliminaries}
\label{sec:triangulations}

This section summarises the results and definitions we need from \cite{Jaco-minimal-2009} in order to carry out the new proofs. The material is repeated here for the benefit of the reader since only a small amount from each subsection in \cite{Jaco-minimal-2009} is needed. The only change is that we have improved the notation in \S\ref{subsec:colouring}.


\subsection{Triangulations}

The notation of \cite{Jaco-0-efficient-2003, Jaco-layered-2006} is used in this paper. A triangulation, $\tri,$ of a 3--manifold $M$ consists of a union of pairwise disjoint 3--simplices, $\widetilde{\Delta},$ a set of face pairings, $\Phi,$ and a natural quotient map $p\co \widetilde{\Delta} \to \widetilde{\Delta} / \Phi = M.$ Since the quotient map is injective on the interior of each 3--simplex, we refer to the image in $M$ of a 3--simplex in $\tri$ as a \emph{tetrahedron} and to its faces, edges and vertices with respect to the pre-image. Similarly for images of 2--simplices, 1--simplices and vertices, which are referred to as \emph{faces,} \emph{edges} and \emph{vertices} in $M,$ respectively. For an edge $e,$ the number of pairwise distinct 1--simplices in $p^{-1}(e)$ is referred to as its \emph{degree}, denoted $d(e).$ If an edge is contained in $\partial M,$ then it is termed a \emph{boundary edge}; otherwise it is referred to as an \emph{interior edge}.


\subsection{Minimal and 0--efficient triangulations}

A triangulation of the closed, orientable, connected 3--manifold $M$ is termed \emph{minimal} if $M$ has no triangulation with fewer tetrahedra. A triangulation of $M$ is termed \emph{0--efficient} if the only embedded, normal 2--spheres are vertex linking. It is shown by the first two authors in \cite{Jaco-0-efficient-2003} that (1) the existence of a 0--efficient triangulation implies that $M$ is irreducible and $M\neq \R P^3,$ (2) a minimal triangulation is 0--efficient unless $M=\R P^3$ or $L(3,1),$ and (3) a 0--efficient triangulation has a single vertex unless $M=S^3$ and the triangulation has precisely two vertices. These facts are used implicitly throughout the article.


\subsection{Layered triangulations of the solid torus and lens spaces}
\label{sec:Layered triangulations of the solid torus and lens spaces}

Layering along a boundary edge is defined in \cite{Jaco-layered-2006} and illustrated in Figure \ref{fig:layering} on the left. We only need this in the restricted setting where  $M= D^2 \times S^1$ is a solid torus and $\tri_\partial$ is a triangulation of $\partial M$ with one vertex. Given a solid torus with a one-vertex triangulation on its boundary, it follows from an Euler characteristic argument that this triangulation has two faces on the boundary meeting along three edges. Suppose $e$ is an edge in $\partial M.$ This is incident to two distinct faces. We say the 3--simplex $\Delta$ is \emph{layered along the (boundary) edge} $e$ if two faces of $\Delta$ are paired, ``without a twist," with the two faces of $\tri_\partial$ incident with $e$. The resulting 3--manifold is homeomorphic with $M$. If $\tri_\partial$ is the restriction of a triangulation $\tri$ of $M$ to $\partial M$, then we obtain a new triangulation of $M.$ 

\begin{figure}[htb]
  \centering{\includegraphics[width=\textwidth]{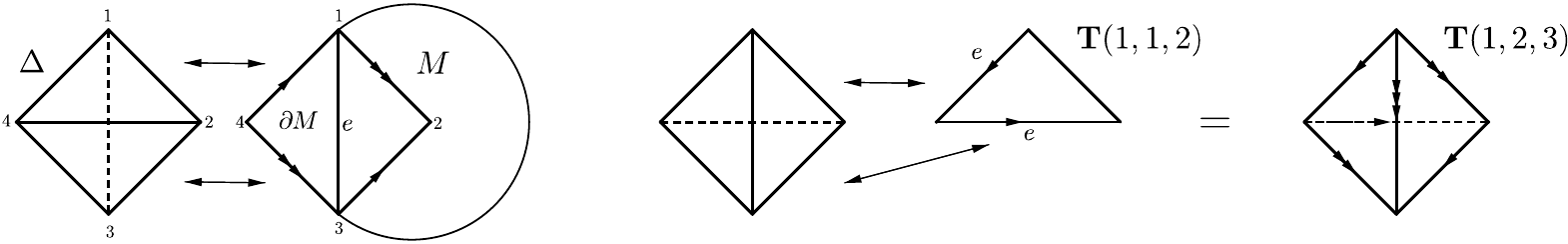}}
    \caption{Left: Layering on a boundary edge. Right: Layered triangulation of the solid torus \label{fig:layering}}
\end{figure}

Starting point for a layered triangulation of a solid torus is the one-triangle triangulation of the M\"obius band shown in Figure~\ref{fig:layering} on the right and considered as a degenerate solid torus and labeled $\lst(1,1,2).$ The first layer is a single tetrahedron layered on the edge labelled $e$ of the M\"obius band in Figure~\ref{fig:layering}, with the two back faces of the tetrahedron identified as indicated by the edge labels. One can then \emph{layer} on any of the three boundary edges (see \cite{Jaco-layered-2006}), giving a layered triangulation of the solid torus with two tetrahedra. 
Inductively, a layered triangulation of a solid torus with $k$ tetrahedra is obtained by layering on a boundary edge of a layered triangulation with $k-1$ tetrahedra, $k \ge 1.$ See also Figure~\ref{fig:tree} for an encoding of these structures by the so-called \emph{L--graph.} A layered triangulation of a solid torus with $k$ tetrahedra has one vertex, $2k+1$ faces, and $k+2$ edges.  Except in the degenerate case, the vertex, three edges, and two faces are contained in the boundary, and there is a special edge in the boundary having degree one; it is called the \emph{univalent edge}.

There is a unique set of unordered nonnegative integers, $\{p,q,p+q\}$, associated with the triangulation on the boundary and determined from the geometric intersection of the meridional slope of the solid torus with the three edges of the triangulation. Two such triangulations with numbers $\{p,q,p+q\}$ and $\{p',q',p'+q'\}$ can be carried to each other via a homeomorphism of the solid torus if and only if the two sets of numbers are identical. There are
three possible ways to identify the two faces in the boundary, each producing a lens space. These identifications can be thought of as ``folding along an edge in the boundary", which is also referred to as ``closing the book" with an edge as the binding. Folding along the edge $p$ gives the lens space $L(2q+p,q)$; along the edge $q$ the lens space $L(2p+q,p)$; and along the edge $p+q$ the lens space $L(\abs{p-q},p)$, see \cite{Jaco-layered-2006}. If the solid torus is triangulated, then after identification one obtains a triangulation of the lens space. If the triangulation of the solid torus is a layered triangulation of the solid torus, then the induced triangulation of the lens space is termed a \emph{layered triangulation} of the lens space. 


\subsection{Edges of low degree in minimal triangulations}
\label{sec:low degree edges}

Given a 1--vertex triangulation of a closed 3--manifold, let $E$ denote the number of edges, $E_i$ denote the number of edges of degree $i,$ and $T$ denote the number of tetrahedra. Then $E = T+1,$ and hence $6 = \sum (6-i)E_i$ since the Euler characteristic of a closed 3--manifold is zero and that of the vertex link is two. It follows that there exist at least two edges of degree at most five. We summarise the results we use; a discussion can be found in \cite{Jaco-minimal-2009}.

\begin{proposition}[Edges of degree one or two, \cite{Jaco-0-efficient-2003, Jaco-minimal-2009}]\label{pro:degree one and two edges}
A minimal and 0--efficient triangulation $\tri$ of the closed, orientable, connected and irreducible 3--manifold $M$ has
\begin{enumerate}
\item[(1)] no edge of order one unless $M=S^3,$ and
\item[(2)] no edge of order two unless $M=L(3,1)$ or $L(4,1).$
\end{enumerate}
\end{proposition}                                                                    

\begin{proposition}[Edges of degree three, \cite{Jaco-0-efficient-2003, Jaco-minimal-2009}]\label{pro:degree three edges}
A minimal and 0--efficient triangulation $\tri$ of the closed, orientable, connected and irreducible 3--manifold $M$ has no edge of order three unless either
\begin{enumerate}
\item[(3a)] $\tri$ contains a single tetrahedron and $M=L(5,2);$ or
\item[(3b)] $\tri$ contains two tetrahedra and $M=L(5,1)$ or $L(7,2);$ or
\item[(3c)] $\tri$ contains, as an embedded subcomplex, the two tetrahedron, layered triangulation $\lst (1,3,4)$ of the solid torus. Moreover, $\tri$ contains at least three tetrahedra and every edge of degree three is an interior edge of such a subcomplex.
\end{enumerate}
Moreover, the triangulations in (3b) are obtained by identifying the boundary faces of $\lst (1,3,4)$ appropriately.
\end{proposition}


\subsection{Intersections of maximal layered solid tori}
\label{sec:Intersections of maximal layered solid tori}

Throughout this section, assume that $M$ is an irreducible, orientable, connected 3--manifold with a fixed triangulation, $\tri.$ Further assumptions are stated. 

\begin{definition}(Layered solid torus)
A \emph{layered solid torus with respect to $\tri$} in $M$ is a subcomplex in $M$ which is combinatorially equivalent to a layered solid torus. Any reference to $\tri$ is suppressed when $\tri$ is fixed. The edges of the layered solid torus $\lst$ are termed as follows: If there exists more than one tetrahedron, then there is a unique edge which has been layered on first, called the \emph{base-edge (of $\lst$).} The edges in the boundary of $\lst$ are called \emph{boundary edges (of $\lst$)} and all other edges (including the base-edge) are termed  \emph{interior edges (of $\lst$).} 
\end{definition}

For every edge in $M,$ we refer to its degree as its $M$--degree, and its degree with respect to the layered solid torus $\lst$ in $M$ is called its $\lst$--degree. The $\lst$--degree of an edge not contained in $\lst$ is zero. The unique edge of $\lst$--degree one is termed a \emph{univalent edge (for $\lst$)}. Clearly, $\lst$--degree and $M$--degree agree for all interior edges of $\lst.$

\begin{definition}(Maximal layered solid torus)
A layered solid torus is a \emph{maximal layered solid torus with respect to $\tri$} in $M$ if it is not strictly contained in any other layered solid torus in $M.$
\end{definition}

A layered triangulation of a lens space with at least three tetrahedra contains precisely two maximal layered solid tori, and they meet in all but two of the tetrahedra in the triangulation. In general, the intersection of maximal layered solid tori is small:

\begin{lemma}\cite{Jaco-0-efficient-2003, Jaco-minimal-2009}\label{lem:intersection of maximal layered solid tori}
Assume that the triangulation $\tri$ is minimal and 0--efficient. If $M$ is not a lens space with layered triangulation, then the intersection of two distinct maximal layered solid tori in $M$ consists of at most a single edge.
\end{lemma}

\begin{lemma}\cite{Jaco-0-efficient-2003, Jaco-minimal-2009}\label{lem:layered lens char 2}
Assume that the triangulation $\tri$ is minimal and 0--efficient, and suppose that $M$ contains a layered solid torus, $\lst,$ made up of at least two tetrahedra and having a boundary edge, $e,$ which has degree four in $M.$ Then either
\begin{enumerate}
\item $\lst$ is not a maximal layered solid torus in $M;$ or
\item $e$ is the univalent edge for $\lst$ and it is contained in four distinct tetrahedra in $M;$ or 
\item $M$ is a lens space with minimal layered triangulation.
\end{enumerate}
\end{lemma}


\subsection{Colouring of edges and dual normal surface}
\label{subsec:colouring}

Throughout this section, let $\tri$ be an arbitrary 1--vertex triangulation of a closed 3--manifold $M,$ and let $\varphi \co \pi_1(M) \to \Z_2$ be a non--trivial homomorphism.
Each edge, $e,$ is given a fixed orientation, and hence represents an element $[e]\in \pi_1(M).$ If $\varphi[e]=0,$ the edge is termed $\varphi$--even, otherwise it is termed $\varphi$--odd. This terminology is independent of the chosen orientation for $e.$ Faces in the triangulation give relations between loops represented by edges. It follows that a tetrahedron falls into one of the following categories, see Figure~\ref{fig:Z2-homology class}:
\begin{itemize}
\item[] Type $\Tq$: A pair of opposite edges are $\varphi$--even, all others are $\varphi$--odd.
\item[] Type $\Tt$: The three edges incident to a vertex are $\varphi$--odd, all others are $\varphi$--even.
\item[] Type $\Te$: All edges are $\varphi$--even.
\end{itemize}

\begin{figure}[htb]
\begin{center}
      \includegraphics[width=.6\textwidth]{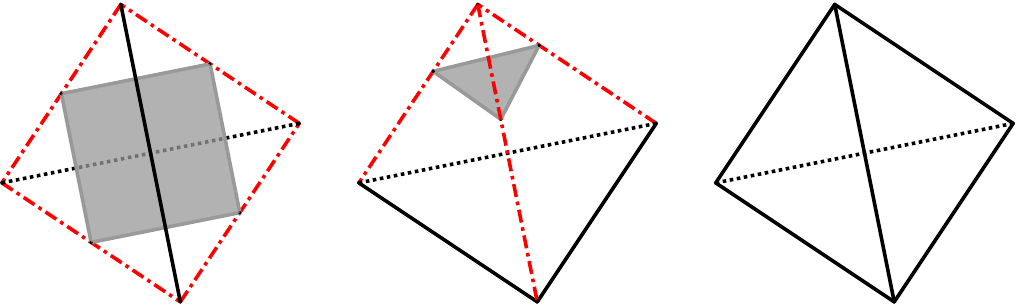}
\end{center}
    \caption{Types of tetrahedra and normal discs of the dual surface}
     \label{fig:Z2-homology class}
\end{figure}

It follows from the classification of the tetrahedra in $\tri$ that, if $\varphi$ is non-trivial, then one obtains a unique normal surface, $S_\varphi(\tri),$ with respect to $\tri$ by introducing a single vertex on each $\varphi$--odd edge. This surface is disjoint from the tetrahedra of type $\Te$; it meets each tetrahedron of type $\Tt$ in a single triangle meeting all $\varphi$--odd edges; and each tetrahedron of type $\Tq$ in a single quadrilateral dual to the $\varphi$--even edges. Moreover, $S_\varphi(\tri)$ is dual to the $\Z_2$--cohomology class represented by $\varphi.$

\begin{lemma}\cite{Jaco-minimal-2009}
Assume that the triangulation $\tri$ contains a single vertex and that all edge loops are coloured using a homomorphism $\varphi \co \pi_1(M) \to \Z_2.$ Then all tetrahedra in a layered solid torus in $M$ are either of type $\Tq$ or type $\Te$, but not both.
\end{lemma}

\begin{definition}[(Types of layered solid tori)]
Assume that the triangulation $\tri$ contains a single vertex and that all edge loops are coloured using a homomorphism $\varphi \co \pi_1(M) \to \Z_2.$ A layered solid torus containing a tetrahedron of type $\Tq$ (respectively $\Te$) is accordingly termed of type $\Tq$ (respectively $\Te$). 
\end{definition}


\section{Rank one}
\label{sec:rank one}

In this section, we give the proof of Theorem~\ref{thm:main 1}.
Suppose $M$ is an orientable, irreducible, closed 3--manifold not homeomorphic with $\R P^3,$  and assume $\varphi$ is a non-trivial element of $H^1(M;\Z_2).$  By hypothesis, $\varphi$ is dual to a surface of non-negative Euler characteristic.

Suppose that $c(M) \le 1 + 2 || \ \varphi \ ||.$ To prove Theorem~\ref{thm:main 1} we need to show that $M$ is a balanced lens space. To simplify notation, we write $g-2=|| \ \varphi \ ||.$ 
We point out that \emph{if} $\varphi$ can be represented by a connected non-orientable surface, \emph{then} $g$ is its genus. Otherwise it is merely a useful symbol for us. Using this definition of $g,$ this gives the  slogan from the introduction:
$$c(M) \le 1 + 2 || \ \varphi \ || = 2g - 3 = c(L(2g,1)).$$
Suppose we have an arbitrary minimal triangulation of $M$ and colour the edges using $\varphi.$ 

\subsection{Promoting triangulations with edge flips}
\label{subsec:promoting}

We first modify the triangulation of $M$ using 4--4 bistellar flips to rule out maximal layered solid tori of special types. 
This follows the argument in \cite{Jaco-minimal-2009} on pages 173--175. We note that this does not change the number of tetrahedra, hence producing other minimal triangulations.

A maximal layered solid torus $\lst$ of type $\Tq$ containing a $\varphi$--even edge of degree three and having the property that all other $\varphi$--even edges in its interior or boundary have degree four is called \emph{supportive}. It follows from Lemma~\ref{lem:layered lens char 2} that the $\varphi$--even boundary edge $e$ of $\lst$ is its univalent edge and contained in four pairwise distinct tetrahedra in the triangulation. As in \cite{Jaco-minimal-2009}, the complex formed by the four tetrahedra around $e$ is termed an octahedron (even though it may not be embedded). There are four cases to consider for the types of tetrahedra ordered cyclically around $e,$ and we start the labelling with the tetrahedron in $\lst$ of type $\Tq$ containing $e.$ The proofs of the following statements are in \cite{Jaco-minimal-2009}, where the notation for the types is 1 instead of $\Tq$; 2 instead of $\Tt$, and 3 instead of $\Te.$
\begin{itemize}
\item[] $(\Tq,\Tq,\Tq,\Tq)$: Any 4--4 bistellar flip on the octahedron results in four tetrahedra of type $\Tt,$ and gives a triangulation with fewer supportive solid tori and the same number of tetrahedra of type $\Te.$
\item[] $(\Tq,\Tt,\Tt,\Tq)$ or $(\Tq,\Tq,\Tt,\Tt)$: Performing a specific 4--4 bistellar flip on the octahedron results in a triangulation with fewer supportive solid tori and the same number of tetrahedra of type $\Te.$
\item[] $(\Tq,\Tt,\Te,\Tt)$: Any 4--4 bistellar flip on the octahedron results in tetrahedra of types $\Tq,\Tq,\Tt,\Tt,$ and hence reduces the number of tetrahedra of type $\Te.$ \end{itemize}
It follows by induction that after finitely many edge flips, we may assume that the triangulation does not contain any supportive solid tori.

\subsection{The fundamental equation}

Let
\begin{itemize}
\item[] $\nq=$ number of tetrahedra of type $\Tq$,
\item[] $\nt=$ number of tetrahedra of type $\Tt$,
\item[] $\ne=$ number of tetrahedra of type $\Te$,
\item[] $\even(\tri)=$ number of $\varphi$--even edges,
item[] $\tilde{\even}(\tri)=$ number of pre-images of $\varphi$--even edges in $\widetilde{\Delta}.$
\end{itemize}


Since the triangulation is 0--efficient, no component of $S_\varphi=S_\varphi(\tri)$ is a sphere or a projective plane. Hence 
$$2-g = - || \ \varphi \ || \ge \chi (S_\varphi)$$
with equality if $S_\varphi$ is taut. If $S_\varphi$ is not geometrically incompressible but admits $k_\varphi$ pairwise disjoint independent compression discs, then this inequality may be written as:
$$2-g = - || \ \varphi \ || \ge \chi (S_\varphi) + 2k_\varphi$$

Consider the complex $K$ spanned by the $\varphi$--even edges. Then taking the normal surface corresponding to twice the normal coordinate of $S_\varphi$ gives the boundary of a small regular neighbourhood of $K.$ Hence
$$2 \chi(S_\varphi) = \chi(\partial K) = 2 \chi(K)  
= 2 \big( 1 - \even + \frac{\nt + 4 \ne}{2} - \ne \big) 
=  2 - 2 \even + {\nt} +  2\ne 
$$
This gives:
$$2(2-g) \ge 2\chi (S_\varphi) + 4k_\varphi= 2 - 2\even + \nt +  2\ne+ 4k_\varphi.$$

The number $c(M)$ of all tetrahedra satisfies $$2g-3 \ge c(M) = \nq + \nt + \ne.$$
We also have
$\sum d \even_d = \tilde{\even} = 2 \nq + 3 \nt + 6 \ne.$
This gives:
\begin{align*}
2 = 4(2-g) + 2(2g-3)
\ge& 4 \chi(S_\varphi) + 8 k_\varphi+ 2 c(M)\\
=& 4 + 8 k_\varphi + \sum (d-4) \even_d + \nt.
\end{align*}
We may assume that $M$ is not the lens space $L(4,1)$ since in this case the conclusion follows by inspection of the triangulation with a single tetrahedron of $L(4,1).$ Hence the smallest degree of an edge in $M$ is 3. Rearranging the inequality gives:
\begin{equation}\label{inequ for min tri}
\even_3 \ge 2 + \sum_{j=5}^{\infty} (j-4) \even_j + 8 k_\varphi + \nt.
\end{equation}

We first remark that if $M$ is a lens space \emph{and} the minimal triangulation is a layered triangulation, then a direct argument shows that $M$ is a balanced lens space. This comes from the fact that in a minimal layered triangulation there are at most two edges of degree three. Hence inequality (\ref{inequ for min tri}) must be an equality, $\even_3=2$ and all other terms must vanish. In particular, apart from the two $\varphi$--even edges of degree 3 all other $\varphi$--even edges have degree 4. The lens spaces satisfying this condition can now be constructed and identified (see \S2.4, Corollary 30 and Figure 6 in \cite{Jaco-minimal-2009}).

Hence we assume that it is not the case that $M$ is a lens space and the minimal triangulation is a layered triangulation. We aim to show that this gives a contradiction. Since the triangulation has at least two even edges of degree three, Proposition~\ref{pro:degree three edges} and Lemma~\ref{lem:intersection of maximal layered solid tori} imply that there are at least two maximal layered solid tori, and that any two maximal layered solid tori meet in at most one edge.

\subsection{Deficit from maximal layered solid tori of type $\Te$}

As in \cite{Jaco-minimal-2009}, one sees that the collection of maximal layered solid tori of type $\Te$ containing a degree 3 edge contribute a \emph{deficit} to the inequality (\ref{inequ for min tri}), in the sense that the contribution to the right hand side from $\sum_{j=5}^{\infty} (j-4) \even_j$ exceeds the contribution to the left hand side $\even_3.$ The argument is as follows. Any maximal layered solid tori of type $\Te$ having an internal edge of degree at least 5,  either \emph{balances} out or contributes a \emph{deficit} to the inequality (\ref{inequ for min tri}). Hence assume we are left with all such tori having one internal edge of degree 3 and all others of degree 4. The edges on the boundary have internal degrees $1, 3, 3+2m,$ where $m \ge 1.$ The edge $e_0$ of internal degree $3+2m$ is the longitude of the layered solid torus. Now consider the collection of all maximal layered solid tori of type $\Te$ meeting in this longitude $e_0$ (some may meet it in their longitude, some in another edge). Whence the degree in $M$ of the longitude is at least 6 (in fact, at least 7). Write this as $5 + k,$ where $k\ge 1.$ Now at most $\frac{k+1}{2}$ layered solid tori can meet in this edge, since no two meet in more than one edge. Hence the contribution to the left hand side is at most $\frac{k+1}{2}$ (and this may be double counting some degree 3 edges), and the contribution to the right hand side is $d(e_0) - 4 = 1+k,$ where $k \ge 0,$ Whence one obtains a deficit since $\frac{k+1}{2} < k+1.$

It follows that there must be at least two maximal layered solid tori of type $\Tq,$ and that from the collection of all such maximal layered solid tori there must be a \emph{gain} of at least two.

\subsection{No gain from maximal layered solid tori of type $\Tq$}

First note that given any maximal layered solid tori of type $\Tq$ having an internal $\varphi$--even edge of degree at least five, then it either \emph{balances} out or contributes a \emph{deficit} to the inequality (\ref{inequ for min tri}). Hence we now restrict to the collection of all maximal layered solid tori of type $\Tq$ containing a $\varphi$--even edge of degree three and having all other interior $\varphi$--even edges of degree four. Since there are no supportive tori (cf. \S\ref{subsec:promoting}), we know that the even boundary edge $e$ of such a torus $\lst$ has degree at least five in $M.$ We call $\lst$ an \emph{almost supportive torus}.

If $e$ is not contained in any other almost supportive torus, then again we have either a balance or a deficit. This shows that there is some almost supportive torus $\lst$ with the property that it shares its even boundary edge with other almost supportive tori.

If $k$ almost supportive tori meet in an even degree edge $e$ of degree $d,$ then $d \ge 2k$ by Lemma~\ref{lem:intersection of maximal layered solid tori}. 
For at least one collection of such tori, there must be a positive contribution. Hence $k \ge 1 + d-4 = d-3$ and so $k+3 \ge d \ge 2k$ for at least one collection. This gives $k\le 3.$ Since we also have $d\ge 5,$ we have the following cases:
\begin{enumerate}
\item $d=6$ and $k=3;$
\item $d=5$ and $k=2.$
\end{enumerate}
Assume that $d=6$ and $k=3.$ It follows that $e$ is contained in six pairwise distinct tetrahedra, with supportive tori alternating around $e.$ All these tetrahedra have to be of type $\Tq.$ Thus, there is a compression disc for $S_\varphi$ associated with $e.$ This compression disc has boundary on the six quadrilateral discs in the tetrahedra containing $e;$ see Figure~\ref{fig:compression}. This contributes $8$ to the right hand side of the equation. We now adjust the contribution of $e$ and the even edges in the boundary of the abstract neighbourhood as follows. Each $\varphi$--even edge not in the interior of an almost supportive torus is given a contribution of 1 (leaving us with at least 4 to still balance). We then perform the compression, resulting in a new normal surface. However, we do not change the types of the tetrahedra. Note that in each tetrahedron associated with the compression, one normal quadrilateral is replaced with two normal triangles and the two even edges of the tetrahedron are adjusted with an extra weight of $1,$ i.e. increasing the contribution of its edge degree by one. The resulting surface is normal and since the triangulation is 0--efficient and $M$ is not $\R P^3,$ no resulting component from the compression is a sphere or projective plane. We also note that if one of the other even edges involved in this process is of the type ($d=6$ and $k=3$) or ($d=5$ and $k=2$), then it now is counted as ($d+1=6+1=7$ and $k=3$) or ($d+1=5+1=6$ and $k=2$) and hence gives a balance or deficit. See Figure~\ref{fig:compression} for details.

\begin{figure}[htb]
  \centerline{\includegraphics[width=.7\textwidth]{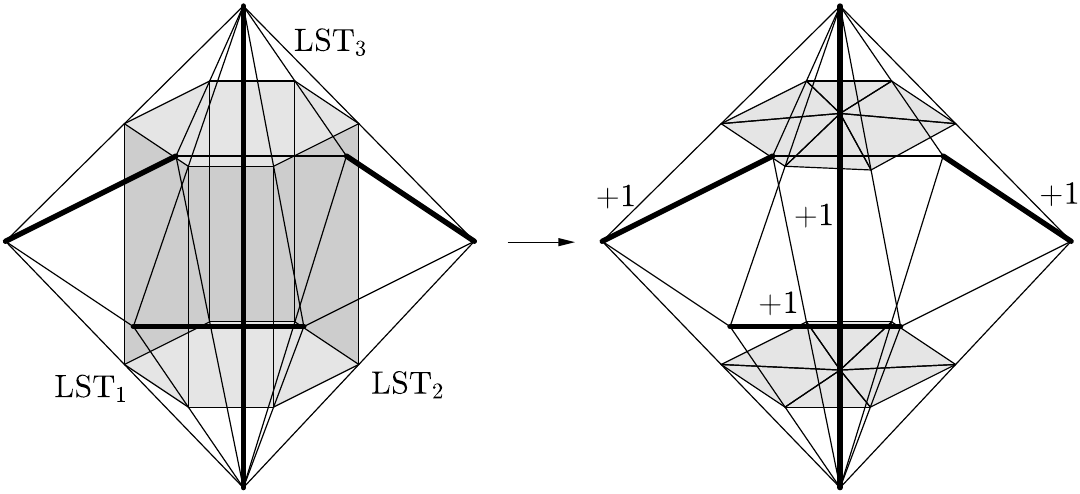}}
  \caption{Compression in the case $d=6$ and $k=3$. A similar situation occurs in the case $d=5$ and $k=2$. \label{fig:compression}}
\end{figure}

Thus, for each edge of type ($d=6$ and $k=3$) we either perform a compression of $S_\varphi$ or, if any of the tetrahedra in its link have already been used in a compression of $S_\varphi$, we do not obtain a gain from it. Also note that this process never splits off spheres or projective planes. 

Hence there must be an edge $e$ of type ($d=5$ and $k=2$) that we have not considered yet. The abstract neighbourhood of $e$ contains three tetrahedra of type $\Delta_q,$ namely two from the two almost supportive tori $\lst_1$ and $\lst_2$, and one meeting each of these in a face. The remaining tetrahedra are either of both of type $\Delta_q,$ or both of type $\Delta_t.$ We first show that they must be distinct.

First consider the case where all tetrahedra are of type $\Delta_q.$ Then the tetrahedra are pairwise distinct since otherwise $\lst_1$ and $\lst_2$ meet in more than one edge. Hence there is a compression disc for $S_\varphi$ associated with the edge $e$, see Figure~\ref{fig:compression}. As in the case of ($d=6$ and $k=3$) this allows us to distribute an extra weight to each of the edges $e$ and the $\varphi$--even edges not in the interior of $\lst_1$ or $\lst_2$ but incident with the collection of five tetrahedra around $e.$

Next, consider the case where three tetrahedra in the abstract edge neighbourhood of $e$ are of type $\Delta_q,$ and the remaining two of type $\Delta_t.$ The tetrahedra of type $\Delta_q$ are pairwise distinct since each of $\lst_1$ and $\lst_2$ contains exactly one of them and the remaining shares a face with each. The tetrahedra of type $\Delta_t$ are distinct since each has only three $\varphi$--odd edges, and the union of their $\varphi$--odd edges contains the four $\varphi$--odd edges in the boundaries of $\lst_1$ and $\lst_2.$ Hence there is a contribution of 2 to the right hand side of inequality (\ref{inequ for min tri}) from the two tetrahedra of type $\Delta_t.$ 

The two tetrahedra of type $\Delta_t$ meet in a face containing three $\varphi$--even edges. If not all of these are of the type ($d=5$ and $k=2$) and have not been visited before, then we obtain a total balance or deficit. Hence assume all three are of type ($d=5$ and $k=2$) and have not been visited before. In this case, we have two normal triangles (from the two tetrahedra of type $\Delta_t$ and nine normal quadrilaterals (three each from the three  $\varphi$--even edges of degree 5) forming a subsurface of $S_\varphi$ that is a thrice punctured sphere. This comes from the fact that no two of the $\varphi$--even edge can now be identified since otherwise two maximal layered solid tori meet in two edges. We replace this thrice punctured sphere by three discs---this corresponds to two compressions and an isotopy on $S_\varphi.$ The resulting surface is again normal: in each tetrahedron of type $\Delta_t$ we replace the unique normal triangles by the other three normal triangles; and in each tetrahedron of type $\Delta_q$ involved in the links of the three $\varphi$--even edges we replace the normal quadrilateral by the two normal triangles meeting the $\varphi$--even edge on the boundary of the layered solid torus. On the level of normal surface theory, this amounts to adding to the normal coordinate of $S_\varphi$ the \emph{edge solutions} corresponding to the three $\varphi$--even edges and subtracting the \emph{tetrahedral solutions} corresponding to the tetrahedra of type $\Delta_t$ (see \cite[\S2.5-2.6]{Luo-angle-2008}). Here, our convention is that edge and tetrahedral solutions have negative quadrilateral coordinates. Then the resulting normal coordinate has non-negative coordinates, the same edge weights modulo 2 as $S_\varphi$ (hence representing the same class) and its Euler characteristic is $\chi(S_\varphi)+4,$ since tetrahedral solutions have formal Euler characteristic 1 and edge solutions have formal Euler characteristic 2. See Figure~\ref{fig:doubleCompression} for an illustration of this process. Now the overall contribution to the left hand side is 6, and the overall contribution to the right hand side is 3 from the degree 5 edges, plus 2 from the type $\Delta_t$ tetrahedra and 16 from the two compressions. Hence we have an overall deficit associated with the three type ($d=5$ and $k=2$) edges. 

\begin{figure}[htb]
  \centerline{\includegraphics[width=\textwidth]{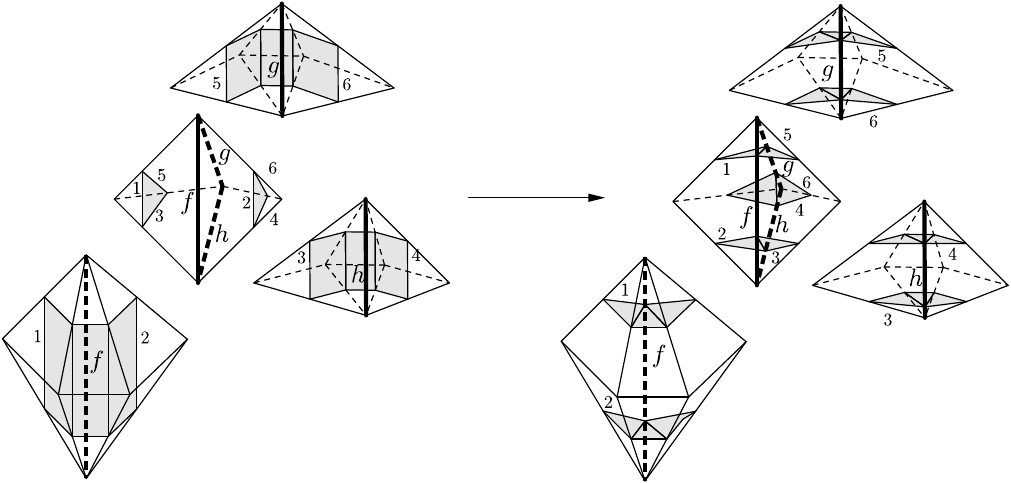}}
  \caption{The double compression in the case $d=5$ and $k=2$. \label{fig:doubleCompression}}
\end{figure}

It follows that no type ($d=5$ and $k=2$) edge can result in a gain, and hence we have a contradiction to our assumptions. This completes the proof of Theorem~\ref{thm:main 1}.


\section{Rank two}
\label{sec:rank 2}

Before we state our second main result, we describe the \emph{twisted squares} from \cite{Luo-combinatorial-2017}. Two pairs of opposite edges in a tetrahedron are the boundary of a quadrilateral properly embedded in the tetrahedron. If opposite pairs of these edges are identified, then one either obtains a pinched projective plane, an embedded Klein bottle or an embedded torus. These are shown in Figure~\ref{fig:cases} and the latter two are called, respectively, a \emph{Klein square} and a \emph{torus square}.
 
 \begin{figure}[htb]
    \centering{\includegraphics[width=.7\textwidth]{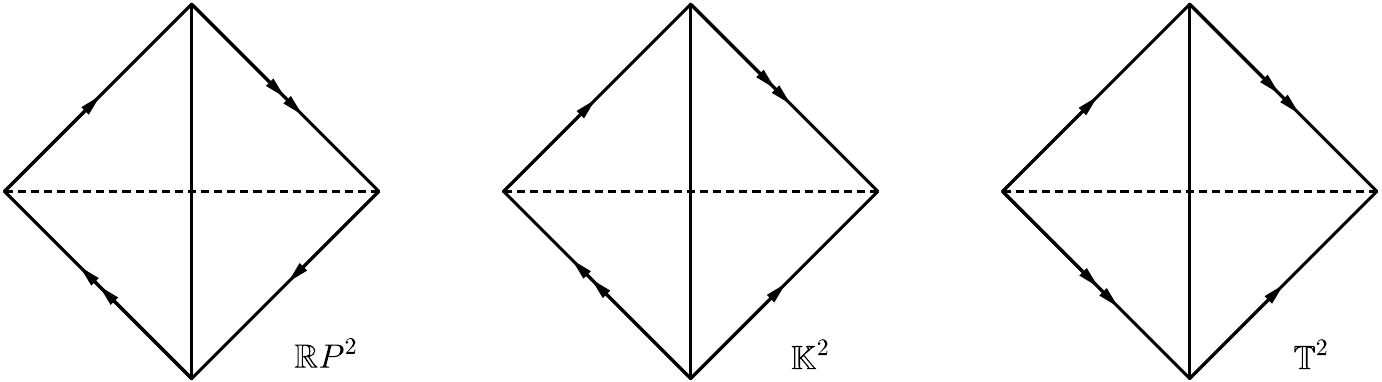}}
    \caption{Tetrahedron with two pairs of opposite edges identified forming a twisted square of type \emph{real projective plane} (left), \emph{Klein bottle} (center) or \emph{torus} (right). \label{fig:cases}}
  \end{figure}

As discussed in the introduction, replacing the use of the lemmata in  \cite{Jaco-Z2-2013} that assume $M$ atoroidal with an improvement of the counting argument analogous to \S\ref{sec:rank one} by taking into account compression discs shows that if $H \le H^1(M;\Z_2)$ is a subgroup of rank two, then 
$$c(M) \ge 2 + \sum_{0 \neq \varphi \in H} ||\;\varphi\;||.$$ 
A detailed argument is given by Nakamura~\cite{Nakamura-complexity-2017}.
Our next main result characterises precisely those manifolds that realise the lower bound.
Theorem~\ref{thm:main 2} is a direct corollary of this result.

\begin{theorem}\label{thm:main 3} 
Let $M$ be a closed orientable, irreducible, connected 3--manifold. Suppose that $H \le H^1(M;\Z_2)$ is a subgroup of rank two.  If $c(M) = 2 + \sum_{0 \neq \varphi \in H} ||\;\varphi\;||,$ then every minimal triangulation contains a tetrahedron with a Klein square or a torus square. Moreover, either
\begin{enumerate}
\item $M$ is a generalised quaternionic space and $H = H^1(M;\Z_2)$; or
\item the rank of $H^1(M;\Z_2)$ is at least three and a Klein square is an incompressible surface in $M;$ or
\item the rank of $H^1(M;\Z_2)$ is at least three and a torus square is an incompressible non-separating surface in $M.$
\end{enumerate}
\end{theorem}

\begin{proof}
Suppose we have a triangulation $M$ realising the bound  $c(M) = 2 + \sum_{0 \neq \varphi \in H} ||\;\varphi\;||$ for $H \le H^1(M;\Z_2)$ subgroup of rank $2$. It follows from \cite{Jaco-Z2-2013} (see also \cite{Nakamura-complexity-2017}) that, given $0 \neq \varphi \in H$, then all tetrahedra of $M$ are of type $\Tq$ and the canonical normal representative $S_\varphi$ is taut.

Denote the edges in the kernel of $\varphi$ by $e_1, \dots , e_k$. We identify vectors in $\Z_2^{c(M)+1}$ with subsets of edges of $M$ in the natural way. It follows that for each $b \in \Z_2^k \subset \Z_2^{c(M)+1}$ we can modify $S_\varphi$ in the following way: for every $b_i = 1$, increase the edge weight of $e_i$ from $0$ to $2$ and leave all other edge weights the same. This gives a normal surface $S_\varphi^b$ still representing $\varphi$. The new surface intersects the tetrahedra of $M$ in three possible ways illustrated in Figure~\ref{fig:modifyS}. We have for the Euler characteristic of $S_\varphi^b$:
$$ \chi (S_\varphi^b) = \chi(S_\varphi) -2 \mathfrak{o} (b) + 2 |b|$$
where $\mathfrak{o} (b)$ denotes the number of octagons in $S_\varphi^b$ and $|b|$ is the $L^1$--norm.

  \begin{figure}[htb]
    \centering{\includegraphics[width=.7\textwidth]{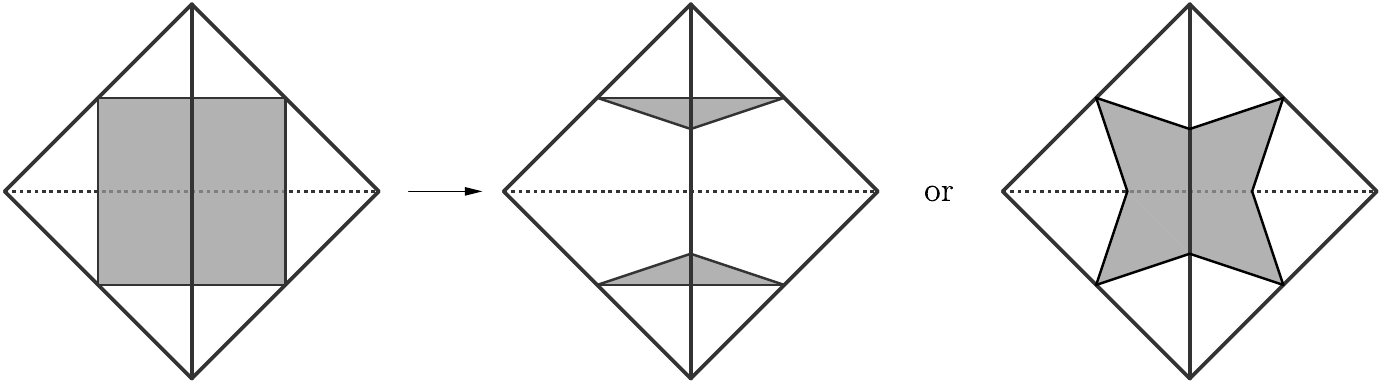}}
    \caption{A quadrilateral disc in $S$ changes to two normal triangles or one normal octagon, depending on whether only one edge opposite the quadrilateral has its weight increased to two or both. \label{fig:modifyS}}
  \end{figure}

Since $S_\varphi$ is taut we have $\chi (S_\varphi) \geq \chi (S_\varphi^b) $ and hence $\mathfrak{o} (b) \geq |b|$. 

Since each edge in the triangulation $M$ is in the kernel for some non-trivial $\varphi \in H$, we know that each unit vector in $\Z_2^{c(M)+1}$ gives at least one octagon. Each unit vector in $\Z_2^{c(M)+1}$ corresponds to one edge and hence there exists at least one tetrahedron with a pair of opposite edges identified with this edge. Since there are strictly more edges than tetrahedra, the pigeonhole principle implies that there exists at least one tetrahedron in which this happens for at least two pairs of opposite edges.

This leaves us with three cases illustrated in Figure~\ref{fig:cases} giving rise to three types of \emph{twisted squares} in the respective tetrahedra. In the first case, the twisted square forms a pinched projective plane in $M$. As in Lemma~15 of \cite{Luo-combinatorial-2017} this can be surgered to give an embedded projective plane which implies that $M \cong \R P^3$, contradicting our hypothesis on the rank of $H$. 

In the second case the twisted square forms an embedded Klein bottle in $M$. This must be non-separating and incompressible and thus there exists a class in $H^1 (M, \Z_2)$ with norm zero. If this class is in $H,$ then one of the quadrilateral surfaces is a Klein bottle and $M$ has a Klein bottle as a 1--sided Heegaard splitting surface, and hence is a prism manifold by \cite{Rubinstein-one-sided-1978}. In this case we apply Proposition~\ref{pro:prism} given below to obtain the first conclusion. If this class is not in $H,$ then we have the second conclusion.
 
In the third case, the twisted square forms an embedded torus $T^2$ in $M.$ Denote $\Gamma = \text{im}(\pi_1(T^2) \to \pi_1(M)).$ This is either equal to $\Z\oplus\Z,$ $\Z,$ or the trivial group. However, two of the edges are contained in the torus and hence $\Gamma$ has a surjection onto $\Z_2\oplus\Z_2$ when mapping $\pi_1(M)$ onto $H_1(M;\Z_2).$ This implies that the torus is incompressible. Moreover, application of the standard half-lives half-dies argument to a component of $M\setminus T^2$ shows that $T^2$ must be non-separating. Hence there is a loop in $M$ transverse to $T^2$ and meeting it in a single point showing that the rank of $H_1(M;\Z_2)$, and hence the rank of $H^1(M;\Z_2)$, is at least 3.
\end{proof}


\section{Examples}
\label{sec:examples}

In this section, we give the examples described in the introduction, and discuss methods to obtain more families.


\subsection{Lens spaces}
\label{subsec:lens spaces}

We use the equations given in \cite{Jaco-minimal-2009}. Suppose that $c(L) = 2 + 2 ||\;\varphi\;||.$
For the canonical surface, we have $\chi(S_\varphi) = 1-\even$ and the number of tetrahedra is $\even + \odd  - 1.$ It also follows by inspection that $||\;\varphi\;|| = - \chi(S_\varphi) = \even -1.$ Whence $c(L) = 2 \even.$

In our fundamental Equation~(\ref{inequ for min tri}), this gives $\sum(d-4)\even_d = 0$ and since $\even_3\le 2,$ we have four possibilities, which are summarised as follows:
\begin{enumerate}
\item $\even_3 =0.$ This implies that only $\even_4 \neq 0.$ This cannot happen.
\item $\even_3 = 1.$ This forces $\even_5=1$ and all other $\varphi$--even edges are of degree 4. This gives the first infinite family of lens spaces.
\item $\even_3 = 2.$ We have ($\even_5=2$ or $\even_6=1$) and all other $\varphi$--even edges are of degree 4. The former case is again not possible. The latter case results in the second infinite family of lens spaces.
\end{enumerate}

We indicate the families of lens spaces satisfying these conditions by adding these new families to the L-Graph in Figure 7 of \cite{Jaco-minimal-2009} and given here in Figure~\ref{fig:tree}. The fraction $p/q, q > p  > 0$, at a vertex of the L-graph corresponds to the one-vertex triangulation on the boundary of the solid torus for the triple $(p,q,p+q)$. The unique minimal path from $p/q$ to $1/1$ describes the unique minimal layered triangulation of the solid torus extending the $(p,q,p+q)$ triangulation on the boundary. At the vertex $p/q$ there are two possible lens spaces determined with minimal layered  triangulations (see Figure~\ref{fig:lgraph} on how to extend the L-graph at $p/q$). We obtain the triangulation of a lens space by what we call folding along either the edge labeled $p$ or the edge labeled $q$ in the boundary of the layered solid torus. In our situation, we require either $p$ or $q$ to be even and fold along the even edge getting the lens space $L(2q+p,q), p$ even, and  $L(2p+q,p)$, $q$ even. 

\begin{figure}[htp]
 \centering{\includegraphics[width=.4\textwidth]{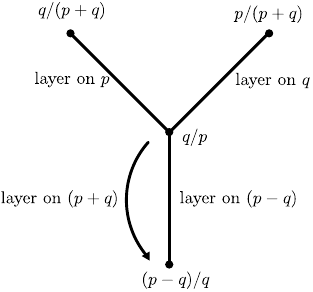}}
 \caption{Extending the L-graph at vertex $p/q$. \label{fig:lgraph}}
\end{figure}

Let $\overline{\even}$  and $\overline{\odd}$  denote the number of even and the number of odd labeled edges in the layered triangulation of the layered solid torus corresponding  to the triple $(p,q,p+q),$ respectively. When viewing the L-graph we define the \emph{deficiency} at the vertex $p/q$ as $d = \overline{\even} - \overline{\odd}$.  Note that when we have either $p$ or $q$ even and we fold along the even edge at $p/q$, the resulting lens space has the same number of even edges as the layered solid torus but one less odd edge. Hence, we satisfy the conditions of Theorem 5 of \cite{Jaco-minimal-2009}, for the lens space formed by folding along an even edge at $p/q$ with $d=1$. Similarly, we satisfy the conditions for the minimal layered triangulation of a lens space to be its minimal triangulation in the two cases described above when we fold along an even edge at $p/q$ with $d=2$. 

The new entries from the results in this paper lie along the infinite subgraph in Figure~\ref{fig:tree} designated by solid edges and obtained by folding along the even edge at a vertex of valence 2 in this solid subgraph. 

\begin{sidewaysfigure}[p]
 \centering{\includegraphics[width=\textwidth]{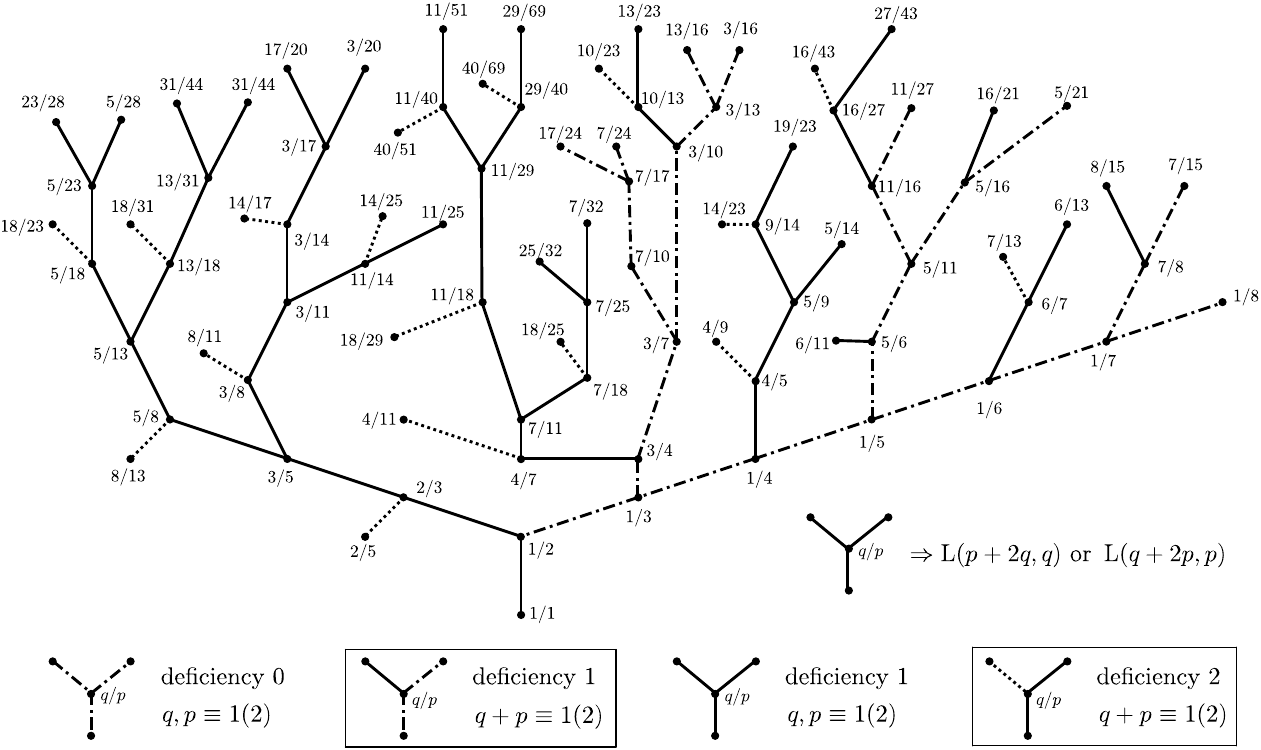}}
 \caption{L-graph of minimal layered solid torus triangulations. \label{fig:tree}}
\end{sidewaysfigure}


\subsection{Prism manifolds}
\label{subsec:prism}

It is observed in \cite{Jaco-Z2-2013} that if a triangulation of $M$ has three taut quadrilateral surfaces representing the non-trivial elements of a subgroup $H$ of $H^1(M;\Z_2)$ of rank 2, then the number of tetrahedra in that triangulation equals $2 + \sum_{0 \neq \varphi \in H} ||\;\varphi\;||,$ and hence is minimal. This characterisation is used in the next result.

\begin{proposition}\label{pro:prism}
Suppose $M$ is a prism manifold with with $H^1(M;\Z_2)$ of rank two. Then a triangulation of $M$ has three taut quadrilateral surfaces if and only if $M$ is the generalised quaternionic space
$$M_k = S^3/Q_{4k} = S^2(\ (1,-1), (2,1), (2,1), (k,1) \ ),$$
for some $k\in 2\N.$ 
\end{proposition}

\begin{proof}
Since $M$ is a prism manifold, one of the quadrilateral surfaces is a Klein bottle, and the triangulation is dual to the 1--sided Heegaard splitting defined by this Klein bottle. Writing the fundamental group of the Klein bottle as $b^{-1} a b = a^{-1},$ a presentation for the fundamental group of the prism manifold is 
$$\langle a, b \ |\ b^{-1}ab = a^{-1}, (b^2)^q a^{2p} = 1 \rangle.$$
Here, the boundary of the singular compression disc for the Klein bottle is given by the curve $(b^2)^q a^{2p}$ on the Klein bottle. The number of tetrahedra in the dual triangulation is $2pq.$ We know that 
$$2pq = 2 + \sum_{0 \neq \varphi \in H} ||\;\varphi\;|| = 2 + 0 + 2n = 2(1+n),$$
where $n$ is the norm of the two surfaces not equal to the Klein bottle. 
The slopes on the singular fiberes are
$(1,-1), (2,1), (2,1), (2p,q)$.

Let $S$ be a pair of pants, i.e. a compact surface with three boundary components and Euler characteristic $-1.$ Then $M$ is obtained from $S \times S^1$ by gluing in three solid tori $D^2 \times S^1.$ The gluing maps are uniquely determined by giving the slopes of the curves on $D^2 \times S^1$ identified with curves $\star \times S^1$ on the boundary of $S \times S^1.$ They are $(2,1)$, $(2,1)$ and $(2p,q-2p)$. A standard calculation with Seifert invariants shows that the resulting manifold is homeomorphic by a fibre preserving map to $M.$ Then the horizontal incompressible surface described by Frohman~\cite{Frohman-one-sided-1986} is obtained by taking $S \times \star$ and extending it over each solid torus by adding the unique one-sided surface of slope $(2,1)$, $(2,1)$ and $(2p,q-2p)$ respectively.

We seek two vertical surfaces with negative Euler characteristic equal to 
$pq-1$. It follows by inspection that there are such surfaces if $q=1.$

Now if $q$ increases in the slope $(2p,q-2p),$ then the norm decreases, hence we only have the desired equality when $q=1.$ This gives precisely the family of generalised quaternionic spaces. This proves the forwards direction. The converse is shown in \cite[\S6]{Jaco-Z2-2013}.
\end{proof}

\begin{remark}
The generalised quaternionic space $M_k$, $k$ even, is also triangulated by the augmented solid torus 
$A(2,1 \mid 2,1 \mid k, 1-k).$ This has one more tetrahedron than $c(M_k).$ In the Pachner graph of all 1--vertex triangulations of $M_k$ connected by 2--3 and 3--2 moves any path from this triangulation to the minimal triangulation is of length at least $2k-11.$ The reason is that the augmented solid torus triangulation contains a layered solid torus consisting of $k-5$ tetrahedra. To remove the self-identification of the core tetrahedron requires at least $k-5$ Pachner moves of type 2--3 in any connecting sequence. Since this increases the number of tetrahedra, and the target triangulation has one tetrahedron fewer, one also needs at least $k-4$ moves of type 3--2 in any connecting sequence. In total, this amounts to at least $2k-9$ Pachner moves. 
\end{remark}

\begin{figure}
  \centerline{\includegraphics[height=4cm]{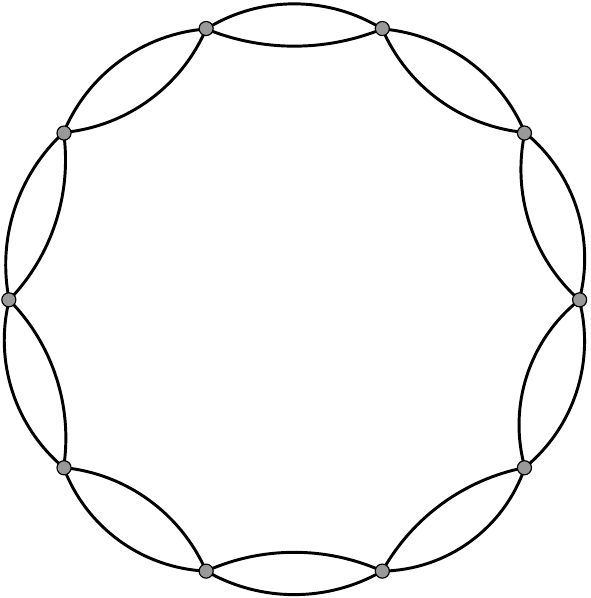} \hspace{2cm} \raisebox{.3cm}{\includegraphics[height=3.2cm]{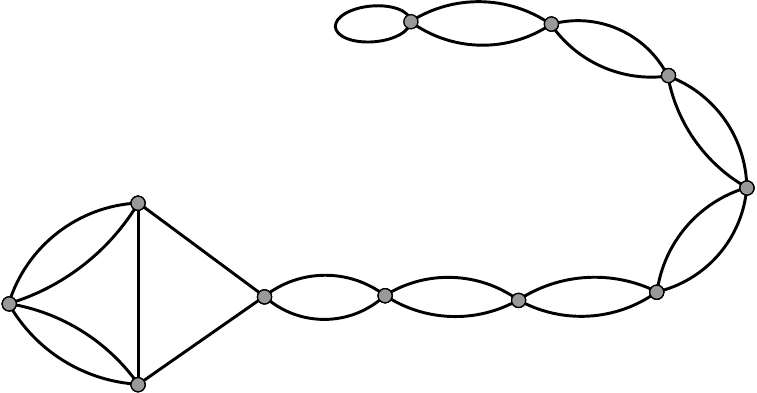}}}
  \caption{Left: Dual graph of layered loop triangulation of $M_{10}$. Right: Dual graph of augmented solid torus triangulation of $M_{10}$.}
\end{figure}

\textbf{Constructions.} We note that the generalised quaternionic spaces are obtained from an augmented solid torus by folding two pairs of its faces to obtain the Seifert invariants $(2,1),$ $(2,1)$ and then adding a layered solid torus that belongs to the subfamily $\lst(1, 1-k, k)$ of deficiency zero from \S\ref{subsec:lens spaces}. Adding a layered solid torus that belongs to the subfamily $\lst(3, 1+6k, 4+6k)$ gives the prism manifolds 
$$P_k = S^2( (1,-1), (2,1), (2,1), (4+6k,1+6k))$$
and these satisfy $c(P_k) = 3 + \sum_{0 \neq \varphi \in H} ||\;\varphi\;||,$ and hence are minimal. More families can be found this way, but we wish to emphasise that not all layered solid tori of deficiency zero result in such prism manifolds with this property.
We explain in Remark~\ref{rem:construct} how to build such examples with \emph{Regina}~\cite{regina}.


\subsection{Small Seifert fibred spaces}
\label{subsec:sSFS}

In this section we determine minimal triangulations for two infinite classes of small Seifert fibred spaces admitting $\widetilde{\SL_2(\R)}$ structures. As in \cite{Jaco-Z2-2013} we rely on work of Waldhausen~\cite{Waldhausen-klasse-1967}, Rubinstein~\cite{Rubinstein-one-sided-1978} and Frohman~\cite{Frohman-one-sided-1986} to determine the norms of the $\Z_2$--homology classes. The discussion in this paper is less detailed than in \cite{Jaco-Z2-2013}. It is a pleasant exercise, using the structure of the triangulations, to give a combinatorial proof of our claims using normal surface theory. 

\textbf{Small Seifert fibred spaces with a horizontal taut surface.}
Let $$M = M_{k,m,n} = S^2( (1,1), (2k+1,1), (2m+1,1), (2n+1,1)),$$ where $k, m$ and $n$ are positive integers. Now $H^1(M; \Z_2) \cong \Z_2$ and we show that the generator $\varphi$ is represented by an incompressible non-orientable horizontal surface of Euler characteristic $-(k+m+n).$ Hence $c(M) \ge 2(k+m+n+1).$ 

The horizontal surface is obtained as follows. On a 2--sphere choose three open discs with pairwise disjoint closures, and denote $S$ the complement of the interiors of the discs. Then $M$ is obtained from $S \times S^1$ by gluing in three solid tori $D^2 \times S^1.$ The gluing maps are uniquely determined by giving the slopes of the curves on $D^2 \times S^1$ identified with curves $\star \times S^1$ on the boundary of $S \times S^1.$ They are $(2k+1, 2k+2),$ $(2n+1, 2n+2)$ and $(2m+1, -2m).$ A standard calculation with Seifert invariants shows that the resulting manifold is homeomorphic by a fibre preserving map to $M.$ Then the horizontal incompressible surface described by Frohman~\cite{Frohman-one-sided-1986} is obtained by taking $S \times \star$ and extending it over each solid torus by adding the unique one-sided surface of slope $(2k+2, 2k+1),$ $(2n+2, 2n+1)$ and $(-2m, 2m+1)$ respectively. These surfaces have respective Euler characteristics $-k,$ $-n$ and $-m+1.$ Adding the Euler characteristic of $S$ gives a total of $-(k+m+n).$ It follows from Frohman~\cite{Frohman-one-sided-1986} that this is the unique incompressible surface representing the nontrivial class in $H^1(M; \Z_2).$

To obtain the equality $c(M) = 2 + 2 ||\;\varphi\;||$ it suffices to exhibit a triangulation with $2(k+m+n+1)$ tetrahedra. The augmented solid torus $A(2k+1,-2k-2 \mid 2m+1,-2m-2 \mid 2n+1, -1)$ is built by first triangulating a prism with three tetrahedra and then gluing the top triangle to the bottom triangle. We denote this $3$-tetrahedron $3$-vertex solid torus by $\tilde{P}$. Every vertex in $\tilde{P}$ is of degree four. See Figure~\ref{fig:prism} for a detailed picture of $\tilde{P}$ and its boundary. We then pinch the vertical edges to obtain a complex with three boundary tori with two triangles each, and only one vertex, denoted by $P$. 

  \begin{figure}[htb]
    \centering{\includegraphics[width=\textwidth]{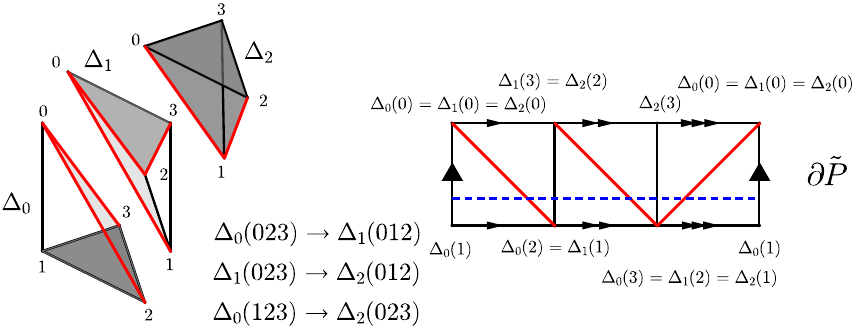}}
    \caption{Left: The solid torus $\tilde{P}$. Right: Boundary complex of $\tilde{P}$, top and bottom vertices become identified in pairs. \label{fig:prism}}
  \end{figure}

The construction of the augmented solid torus is finished by gluing layered solid tori of type $(1,2k+1,2k+2)$, $(1,2m+1,2m+2)$ and $(1,2n,2n+1)$ along the boundary components of $P$.  For each of the boundary components, the even-order edge is glued to the horizontal edge, the edge of order one is glued to a vertical edge, and the remaining edge is glued to the remaining diagonal edge. See Figure~\ref{fig:prism} for details and see \cite{regina} for a more elaborate discussion of augmented layered solid tori. From the description of $M$ above, this shows that the resulting identification space is homeomorphic with $M$. Since the triangulation has $2(k+m+n+1)$ tetrahedra, we obtain $c(M) \le 2(k+m+n+1) = 2 + 2 ||\;\varphi\;||,$ which forces equality.

We conclude this example by describing a normal surface approach to determining the norm. Any normal representative  of the unique non-zero homology class $\varphi \in H_2 (M; \Z_2)$ intersects the horizontal edges of $P$ (i.e., the three edges belonging to the top and bottom triangle of the prism) an even number of times, and all other edges of $P$ an odd number of times. It can be shown that the only least weight least genus normal surface representative of $\varphi$ is the canonical normal representative $S_\varphi.$ The intersection of $S_\varphi$ with the boundary faces of $P$ os shown as a dashed line in the induced boundary of $\partial P$ in Figure~\ref{fig:prism}.

\textbf{Small Seifert fibred spaces with three vertical taut surfaces.}
Let $$M= M'_{k,m,n} = S^2( (1,-1), (2k+2,1), (2m+2,1), (2n+2,1)),$$ where $k, m$ and $n$ are positive integers. As above, one can show that this is triangulated by the \emph{augmented solid torus} $A(2k+2,-1 \mid 2m+2,-1 \mid 2n+2, -1)$, which has $2k+2m+2n+3$ tetrahedra. Hence $c(M) \le 2k+2m+2n+3.$

One has $H^1(M, \Z_2) \cong \Z_2\oplus \Z_2.$ The manifold is obtained similar to above by gluing solid tori to $S\times S^1,$ this time using slopes $(2k+2, 1),$ $(2m+2, 1)$ and $(2n+2, -2n-1).$ A vertical one-sided surface is obtained by taking an arc $\alpha$ in $S$ connecting distinct boundary component of $S$ and attaching to the annulus $\alpha \times S^1$ one-sided incompressible surfaces in the tori met by the boundary of the annulus. This gives surfaces of respective Euler characteristics $-k-m,$ $-m-n$ and $-n-k$ representing the non-trivial $\Z_2$--classes. As in \cite{Jaco-Z2-2013} one argues that these are norm minimising, hence giving 
$c(M) = 3 + \sum_{0\neq \varphi} ||\;\varphi\;||.$

\begin{remark}
  \label{rem:construct}
 To construct a minimal triangulation of $M_{k,m,n}$ (resp. $M'_{k,m,n}$) with \emph{Regina} \cite{regina}, one creates a new 3--manifold by selecting type of triangulation ``Seifert fibred space over $2$-sphere'' and passing the (differently normalised) parameters $(2k+1,1) (2m+1,1) (2n+1,2n+2)$ (resp. $(2k+2,1) (2m+2,1) (2n+2,-2n-1)$). Regina then generates an augmented solid torus of type $A(2k+1,-1 \mid 2m+1,-1 \mid 2n+1, -4n-3)$ (resp. $A(2k+2,-1 \mid 2m+2,-1 \mid 2n+2, -1)$. These minimal triangulations are in general not unique, e.g. another minimal triangulation of  $M_{k,m,n}$ is obtained as $A(2k+1,-2k-2 \mid 2m+1,-2m-2 \mid 2n+1, -1).$ 
\end{remark}


\bibliographystyle{plain}
\bibliography{minimal}


\address{William Jaco\\Department of Mathematics, Oklahoma State University, Stillwater, OK 74078-1058, USA\\{william.jaco@okstate.edu}\\-----}

\address{J. Hyam Rubinstein\\School of Mathematics and Statistics, The University of Melbourne, VIC 3010, Australia\\
{joachim@unimelb.edu.au}\\----- }

\address{Jonathan Spreer\\Institut f\"ur Mathematik, Freie Universit\"at Berlin, Arnimallee 2, 14195 Berlin, Germany\\ 
{jonathan.spreer@fu-berlin.de\\-----}}

\address{Stephan Tillmann\\School of Mathematics and Statistics F07, The University of Sydney, NSW 2006 Australia\\{stephan.tillmann@sydney.edu.au}}

\Addresses
                                                      
\end{document}